\documentclass[12pt]{amsart}
\usepackage{amssymb}
\usepackage{amscd}
\usepackage{eucal}
\usepackage{mathrsfs}
\usepackage{a4wide}
\usepackage[all]{xy} 

\def\pd{\mathrm{pd}}
\def\density{\mathrm{d}}
\def\cel{\mathrm{c}}
\def\con{\subseteq}
\def\from{\colon}
\def\cont{\mathfrak{c}}
\def\CL#1{\overline{#1}}
\newcommand\GCH{\ensuremath{\mathsf{GCH}}}
\newcommand\WGCH{\ensuremath{\mathsf{WGCH}}}
\def\integers{\mathbb{Z}}
\def\reals{\mathbb{R}}

\def\arhang{Arhangel'skii}
\def\mathcal#1{\mathscr{#1}}
\newcommand{\cf}{\operatorname{cf}}

\newtheorem{theorem}{Theorem}[section]
\newtheorem{corollary}[theorem]{Corollary}
\newtheorem{lemma}[theorem]{Lemma}
\newtheorem{proposition}[theorem]{Proposition}
\theoremstyle{definition}

\newtheorem{example}[theorem]{Example}
\theoremstyle{remark}

\numberwithin{equation}{section}

\makeatletter
\def\range{\mathop{\operator@font range}\nolimits}
\makeatother

\begin{document}

\author{Istvan Juh\'asz}
\address{Alfr\'ed Renyi Institute of Mathematics, Hungarian Academy of Sciences}
\email{juhasz@renyi.hu}
\author{Jan van Mill}
\address{KdV-institute for Mathematics, University of Amsterdam}
\email{j.vanMill@uva.nl}
\author{Lajos Soukup}
\address{Alfr\'ed Renyi Institute of Mathematics, Hungarian Academy of Sciences}
\email{soukup@renyi.hu}
\author{Zolt\'an Szentmikl\'ossy}
\address{E\H{o}tv\H{o}s University of Budapest}
\email{szentmiklossyz@gmail.com}

\title[pd-examples]{Connected and/or topological group pd-examples}

\date{\today}



\begin{abstract}
The pinning down number $\pd(X)$ of a topological space $X$ is the smallest cardinal $\kappa$ such that for every neighborhood assignment $\mathcal{U}$ on $X$
there is a set of size $\kappa$ that meets every member of $\mathcal{U}$. Clearly, $\pd(X) \le \density(X)$ and we call $X$ a pd-example if $\pd(X) < \density(X)$.
We denote by $\mathbf{S}$ the class of all singular cardinals that are not strong limit.
It was proved in ~\cite{JuhaszSoukupSzentmiklossy16} that TFAE:

\begin{enumerate}
\item $\mathbf{S} \ne \emptyset$;
\item there is a 0-dimensional $T_2$ pd-example;
\item there is a $T_2$ pd-example.
\end{enumerate}

The aim of this paper is to produce pd-examples with further interesting topological properties like connectivity or being a topological group
by presenting several constructions that transform given pd-examples into ones with these additional properties.

We show that $\mathbf{S} \ne \emptyset$ is also equivalent to the existence of a connected and locally connected $T_3$ pd-example,
as well as to the existence of an abelian $T_2$ topological group pd-example.

However, $\mathbf{S} \ne \emptyset$ in itself is not sufficient to imply the existence of a connected $T_{3.5}$ pd-example. But if
there is $\mu \in \mathbf{S}$ with $\mu \ge \mathfrak{c}$ then there is an abelian $T_2$ topological group (hence $T_{3.5}$) pd-example
which is also arcwise connected and locally arcwise connected.
Finally, the same assumption $\,\mathbf{S} \setminus \mathfrak{c} \ne \emptyset\,$ even implies that there is
a locally convex topological vector space pd-example.
\end{abstract}

\thanks{The research and preparation on this paper was supported by NKFIH grant no. K113047. It derives from the authors' collaboration at the Renyi Institute in Budapest in
the spring of 2016. The second-listed author is pleased to thank the Hungarian Academy of Sciences for generous support (distinguished scientists program) and the R\'enyi Institute for excellent conditions and generous hospitality.}

\maketitle

\section{Introduction}\label{introduction}

\bigskip

The \emph{pinning down number} $\pd(\mathcal{A})$ of a family of sets $\mathcal{A}$ is defined to be the smallest
cardinality of a set that intersects every non-empty member of $\mathcal{A}$. For a topological space $X$, the \emph{pinning down number of $X$}, abbreviated $\pd(X)$, is defined as follows:
$$
    \pd(X) = \sup\big\{\pd(\mathcal{U}) : \mathcal{U} = \{U_x : x \in X\} \mbox{ is a neighbourhood assignment on }X \big\}.
$$
Clearly, $\cel(X) \le \pd(X) \le \density(X)$. Here $\cel(X)$ and $\density(X)$ denote the cellularity and density of $X$, respectively.

The cardinal function pd$(X)$ has been introduced recently, under different names, by Aurichi and Bella in \cite{AurichiBella15} and independently
by Banakh and Ravsky in \cite{BanakhRavsky16}. The latter showed among other things that if $|X|< \aleph_\omega$, then $\pd(X) = \density(X)$, \cite[Theorem 5.2]{BanakhRavsky16}.

The \emph{Weak Generalized Continuum Hypothesis} (abbreviated \WGCH) is the statement that $2^\kappa$ is a finite successor of $\kappa$ for any cardinal $\kappa$,
i.e. in symbols: $(\forall\,\kappa)(2^\kappa < \kappa^{+\omega})$. Clearly, \WGCH\ is equivalent to the statement that
every singular cardinal is strong limit. Answering some problems raised in \cite{BanakhRavsky16}, the following was proved recently in Juh\'asz, Soukup and Szentmikl\'ossy~\cite[Theorem 1.2]{JuhaszSoukupSzentmiklossy16}:

\begin{theorem}\label{maintheorem}
The following statements are equivalent:
\begin{enumerate}
\item \WGCH;
\item $\density(X) = \pd(X)$ for every $T_2$ space $X$;
\item $\density(X) = \pd(X)$ for every 0-dimensional $T_2$ space $X$.
\end{enumerate}
\end{theorem}

Let us call the topological space $X$ a \emph{pd-example} if $\pd(X) < \density(X)$
and denote by $\mathbf{S}$ the class of all singular cardinals that are not strong limit.
Thus Theorem \ref{maintheorem} says that (0-dimensional) $T_2$ pd-examples exist
iff $\mathbf{S} \ne \emptyset$.
The aim of this paper is to examine what is needed to obtain $T_2,\,T_3$, or $T_{3.5}$ pd-examples
with the additional properties of connectivity and/or homogeneity. Here is our first main result:

\begin{theorem}\label{maintheorem1}
TFAE with the negation of the statements in Theorem~\ref{maintheorem},
in particular with $\mathbf{S} \ne \emptyset$:
\begin{enumerate}
\item [(I)] There is a connected and locally connected $T_3$ pd-example;
\item[(II)] there is an abelian $T_2$ topological group pd-example.
\end{enumerate}
\end{theorem}

It is conspicuous that this result does not provide us with a connected $T_{3.5}$ pd-example.
In fact, we shall show that $\mathbf{S} \ne \emptyset$ is not sufficient for this.
Our second main theorem clarifies the reason behind this fact.

\begin{theorem}\label{maintheorem2}
If there is $\mu \in \mathbf{S}$ with $\mu \ge \mathfrak{c}$ then there is an abelian $T_2$ topological group pd-example
which is also arcwise connected and locally arcwise connected.
Moreover, if Shelah's strong hypothesis (in short: SSH) holds then the existence of a connected $T_{3.5}$ pd-example
implies that there is $\mu \in \mathbf{S}$ with $\mu \ge \mathfrak{c}$. Consequently, under SSH the converse of
the first statement is also valid.
\end{theorem}

Of course, any locally convex topological $\mathbb{R}$-vector space is
an abelian $T_2$ topological group which is arcwise connected and locally arcwise connected.
Thus our third main result is a strengthening of Theorem \ref{maintheorem2}.
Still, we decided to present the two results separately because their proofs are based
on two very different constructions that are both interesting in their own rights.

\begin{theorem}\label{maintheorem3}
If there is $\mu \in \mathbf{S}$ with $\mu \ge \mathfrak{c}$ then there is a
pd-example which is a locally convex topological $\mathbb{R}$-vector space.
\end{theorem}

\bigskip

\section{Preliminary observations and definitions}

\bigskip

If $X$ is any space then $\tau(X)$ denotes its topology and $\tau^+(X)=\tau(X)\setminus \{\emptyset\}$.
We also put
$\triangle(X) = \min\{|G| : G\in \tau^+(X)\}$ and call $X$ \emph{neat} if $\triangle(X)=|X|$.

First in this section we collect some facts that will be important later when we calculate the pinning down numbers of various spaces.

\begin{lemma}\label{nuldelemma}
Let $X$ be a space and let $\mathcal{U}\con \tau^+(X)$ with $|\mathcal{U}| \le {\triangle(X)}$. Then there is a neighborhood assignment on $X$ that contains $\mathcal{U}$  in its range,
hence $\pd(\mathcal{U}) \le \pd(X)$.
\end{lemma}

\begin{proof}
Since every nonempty open subset of $X$ has size at least $\triangle(X)$ and $|\mathcal{U}|\le \triangle(X)$, by a trivial transfinite recursion, it is possible to pick for every $U\in \mathcal{U}$ a point $x_U\in U$ such that for distinct $U,V\in \mathcal{U}$ the points $x_U$ and $x_{V}$ are different. Put $A=\{x_U: U\in \mathcal{U}\}$. Then $x_U\mapsto U$ is a `partial' neighborhood assignment on $A$ which can be extended to
the required global neighborhood assignment on $X$.
\end{proof}

\begin{corollary}\label{derdelemma}
Let $X$ be a neat space with $\pd(X) \ge \omega$ and $\kappa \le \pd(X)$ be a cardinal such that $|X^\kappa|=|X|$. Then $\pd(X^\kappa) = \pd(X)$.
\end{corollary}

\begin{proof}
Since $X$ is a continuous image of $X^\kappa$, we clearly have $\pd(X) \le \pd(X^\kappa)$. To prove the converse inequality, let $U\from X^\kappa\to \tau(X^\kappa)$ be any neighborhood assignment.
We may assume that for every $y\in X^\kappa$ there are a finite $F(y)\con \kappa$ and for all $\alpha\in F(y)$ open sets $U^y_\alpha \in \tau(X)^+$ such that
$$
    y\in \{z\in X^\kappa : (\forall\, \alpha\in F(y)) (z_\alpha\in U^y_\alpha)\} = U(y).
$$
Now put
$$
    \mathcal{U}=\{U^y_\alpha : y\in X^\kappa, \alpha < \kappa\}.
$$
Then $|\mathcal{U}| \le |X|^\kappa = |X| = \Delta(X)$, hence by Lemma~\ref{nuldelemma} we have $\pd(\mathcal{U}) \le \pd(X)$.
Let $A$ be a subset of $X$ of size $\pd(X)$ that meets every member of $\mathcal{U}$. Fix $p\in X$ and put
$$
    \sigma = \{y\in X^\kappa : (\exists\, F\in [\kappa]^{<\omega}) (\forall \alpha\in F)(y_\alpha\in A)\, \& \, (\forall\, \alpha \not\in F)(y_\alpha = p)\}.
$$
Then $|\sigma| = \big|[A]^{< \omega}\big| = |A| = \pd(X)$ and for every $y\in X$ we have $\sigma\cap U(y)\not=\emptyset$. Hence we are done.
\end{proof}

The following obvious result will play an essential role in our constructions.

\begin{corollary}\label{X^n}
For every neat $T_2$ space $X$ and for every $0 < n < \omega$ we have
$\pd(X^n) =\pd(X)$.
\end{corollary}

But can the assumption of neatness of $X$ dropped here? Clearly yes if $\pd(X) = \density(X)$,
hence a counterexample is a $T_2$ pd-example.
And in fact, the existence of a $T_2$ pd-example, that we know is equivalent to $\mathbf{S} \ne \emptyset$, does yield a counterexample.

\begin{example}\label{eerstevoorbeeld}
If $\mathbf{S} \ne \emptyset$ then there is a 0-dimensional $T_2$  space $X$ such that $\pd(X^2) > \pd(X)$.
\end{example}

\begin{proof}
If $\mu \in \mathbf{S}$ then there is a cardinal $\lambda$ satisfying $\cf(\mu) \le \lambda < \mu$ and $2^\lambda > \mu$.
The construction theorem ~\cite[Theorem 3.3]{JuhaszSoukupSzentmiklossy16}, in fact a simplified version of it,
then yields a 0-dimensional $T_2$ space $Y$ such that
$\pd(Y) \le \lambda < \density(Y) = \mu$ and $w(Y) \le 2^\lambda$.

Let $Z$ denote the Cantor cube of weight $2^\lambda$. We claim that $\pd(Y \times Z) = \mu$. Then the topological sum $X = Y \oplus Z$ is the space we are looking for.
Indeed, $\pd(Y) \le \lambda$ and $\pd(Z) = \density(Z) \le \lambda$ obviously imply $\pd(X) \le \lambda$, while $\pd(X^2) \ge \mu = \pd(Y \times Z)$ holds
because $Y \times Z$ is an open subspace of $X^2$.

Now, let $\mathcal{B}$ be a base of $Y$ of size $\le 2^\lambda$ and consider the collection
$$
   \mathcal{V} = \{B \times Z : B \in \mathcal{B}\} \subset \tau^+(Y \times Z).
$$
Then $|\mathcal{V}| = |\mathcal{B}| \le 2^\lambda$  and we have $\Delta(Y \times Z) = \Delta(Z) = 2^\lambda$, hence
Lemma~\ref{nuldelemma} implies $\pd(Y \times Z) \ge \pd(\mathcal{V})$. But if $A \subset Y \times Z$ and $|A| < \mu$
then the projection $P$ of $A$ to $Y$ is not dense in $Y$, hence there is a $B \in \mathcal{B}$ with $B \cap P = \emptyset$.
Consequently, $(B \times Z) \cap A = \emptyset$ as well, hence we have $\pd(\mathcal{V}) \ge \mu$.
\end{proof}

Note that any neat $T_2$ space that is not a singleton is infinite, consequently if $X$ is such a space then $\density(X^\omega) = \density(X)$.
On the other hand, our next example shows that $\pd(X) < \pd(X^\omega)$ may hold for a neat $T_2$ space $X$.
This also shows that the assumption $|X^\kappa|=|X|$ cannot be dropped from Lemma \ref{derdelemma} either.
Again, this space must be a $T_2$ pd-example.

\begin{example}
It is consistent that there is a neat 0-dimensional $T_2$ space $X$ such that $\pd(X) < \pd(X^\omega)$.
\end{example}

\begin{proof}
By \cite[Theorem 3.3]{JuhaszSoukupSzentmiklossy16} again, it is consistent that there is a neat 0-dimensional $T_2$ pd-example $X$ such that $\aleph_\omega = |X| < w(X) \le \mathfrak{c}$.
But then $X^\omega$ is not a pd-example because $w(X^\omega) = w(X) \le \mathfrak{c} = \Delta(X^\omega)$, hence $\pd(X^\omega) = \density(X^\omega) = \density(X) > \pd(X)$.
\end{proof}

It is obvious that, for any space $X$, if $Y$ is dense open in $X$ then $\density(X) = \density(Y)$.
It is natural to ask if this also holds with $\pd$ instead of $\density$. Of course, if $X \supset Y$ form a counterexample to this then $Y$
must be a pd-example. Even though it will not be used in our later constructions, we present below
such an example that is by no means trivial.

\begin{example}
It is consistent that there is a $T_{3.5}$ space $X$ with a dense open subspace~$Y$ such that $\pd(Y) < \pd(X)$.
\end{example}

\begin{proof}
By \cite[Theorem 3.3]{JuhaszSoukupSzentmiklossy16}, it is consistent that there is a 0-dimensional neat $T_{2}$ space $Z$ such that $\pd(Z)=\omega < \density(Z)$,
and by \cite[Theorem 4.9]{JuhaszSoukupSzentmiklossy16} we have $|Z|=\triangle(Z) < 2^{\pd(Z)} = \cont$.
Hence $Z$ has at most $2^\cont$ many cozero-sets. Let $\mathcal{B}$ denote the collection of all nonempty cozero-sets in $Z$.

Consider the space $Y=\omega\times Z$ and its \v{C}ech-Stone compactification $\beta Y$ and then put
$$
    X = Y \cup {\big \{}p\in \beta Y : p\not\in \bigcup_{n<\omega} \CL{\{n\}\times Z}{\big \}}.
$$
(Here closures are taken in $\beta Y$). Then each $\{n\}\times Z$ is clopen in $X$, hence $Y$ is dense open in $X$. Clearly, $\pd(Y)=\omega$ and we claim that $\pd(X) > \omega$.

Striving for a contradiction, assume that $\pd(X) =\omega$. For every function $f\in \mathcal{B}^\omega$, let $V(f)$ be the largest open subset of $\beta Y$ such that $V(f) \cap Y = \bigcup_{n<\omega} \{n\}\times f(n)$. We claim that $V(f)\cap X$ has size at least $2^\cont$. Indeed, for every $n$ let $Z_n$ be a nonempty zero-set in $Z$ that is contained in $f(n)$. Then $\bigcup_{n<\omega} \{n\}\times Z_n$ is a zero-set in $Y$ that is disjoint from the zero-set $Y\setminus V(f)$. Hence these two zero-sets have disjoint closures in $\beta Y$. For every $n$, pick $p_n\in Z_n$. Then $V(f)$ contains the closure of the discrete and closed subset $P = \{(n,p_n) : n <\omega\}$ of $Y$. Observe that $\CL{P}$ is homeomorphic to $\beta\omega$ and has therefore size $2^\cont$. Finally, $\CL{P}\setminus P$ is contained in $X$ and so indeed $V(f)\cap X$ has size at least $2^\cont$.

By a simple transfinite induction we can consequently pick points $z(f) \in V(f) \cap X$ such that for all $f,g\in \mathcal{B}^\omega$, if $V(f) \cap X \not= V(g)\cap X$, then $z(f)\not=z(g)$. Since $\pd(X) = \omega$, then there is a countable subset $A$ of $X$ which meets all members from $\{V(f)\cap X : f\in \mathcal{B}^\omega\}$. But the collection $\{V(f) \cap X: f\in \mathcal{B}^\omega\}$ is a basis for $X$, which implies that $X$ and its dense open subspace $Y$ is separable, contradicting that $Z$ is not separable.
\end{proof}

We end this section with a few simple but useful results.

\begin{lemma}\label{eerstelemmaA}
Let $\mathcal{A}$ be any cover of the space $X$. Then $\pd(X) \le \sum \{\pd(A) : A \in \mathcal{A}\}$.
\end{lemma}

\begin{proof}
Let $U\from X\to \tau(X)$ be a neighborhood assignment on $X$. Then, for every $A\in \mathcal{A}$, the function $V_A\from A\to \tau(A)$ defined by $V_A(a) = U(a)\cap A$ is a neighborhood assignment on $A$,
hence $\pd(\{U(a) : a\in A\}) \le \pd(A)$. The rest is obvious.
\end{proof}

It is immediate from Lemma \ref{eerstelemmaA} that $\pd(X \times Y) \le |Y| \cdot \pd(X)$ holds
for any product $X \times Y$, hence $\pd(X \times Y) = \pd(X)$ if $|Y| \le \pd(X)$.
The following result yields a similar implication in which $\density(Y) \le \pd(X)$ is sufficient
instead of $|Y| \le \pd(X)$.

\begin{lemma}\label{vrijdageen}
Let $X$ and $Y$ be spaces such that $\density(Y) \le \pd(X)$ and $|Y|\le \Delta(X) = |X|$. Then $\pd(X\times Y) = \pd(X)$.
\end{lemma}

\begin{proof}
Let $(x,y)\mapsto U_{(x,y)} \times V_{(x,y)}$ be a neighborhood assignment on $X\times Y$ such that $U_{(x,y)}$ (resp. $V_{(x,y)}$) is an open neighborhood of $x$ in $X$ (resp. $y$ in $Y$).
Consider the collection $\mathcal{U} = \{U_{(x,y)} : x\in X, y\in Y\}$. Then, by our assumptions, $|\mathcal{U}| \le |X|=\triangle(X)$, hence by Lemma~\ref{nuldelemma}
there is a set $B\in [X]^{\le\pd(X)}$ that meets every member of $\mathcal{U}$. Let $D\con Y$ be dense such that $|D|\le \pd(X)$.
Then $(B\times D) \cap (U_{(x,y)} \times V_{(x,y)}) \ne \emptyset$ for all $(x,y)\in X\times Y$ and $|B\times D| \le \pd(X)$,
hence we conclude $\pd(X\times Y) \le \pd(X)$. But $\pd(X\times Y) \ge \pd(X)$ holds because $X$ is the continuous image of $X \times Y$.
\end{proof}

The following proposition from \cite{JuhaszSoukupSzentmiklossy16} is added here because it will be used quite frequently.

\begin{proposition}[{\cite[Lemma 2.2]{JuhaszSoukupSzentmiklossy16}}]\label{neat}
Every pd-example $X$ has a neat open subspace $Y$ that is also a pd-example.
\end{proposition}

\bigskip

\section{Proof of Theorem~\ref{maintheorem1} part (I): connected pd-examples}\label{functorseen}

\bigskip

In~\cite{deGroot69}, de Groot associated to every $T_1$ space $X$ a certain extension $\lambda X$ which he called the {\em superextension} of $X$. We will briefly describe his construction.

A system of sets $\mathcal{L}$ is called a \emph{linked system} if any two of its members meet. A \emph{maximal linked system} (or \emph{mls}) on $X$ is a system of closed subsets of $X$ which is maximal with respect to being linked. We denote the collection of all mls's on $X$ by $\lambda X$. For any $A\con X$ we write
$$
    A^+ = \{\mathcal{M}\in \lambda X : (\exists\, M\in \mathcal{M}) (M\con A)\}.
$$
We then take the collection
$$
    \{A^+ : A\mbox{ is closed in $X$}\}
$$
as a \emph{closed} subbase for the topology of $\lambda X$. With this topology, $\lambda X$ is a (super)compact $T_1$-space
that contains $X$ as a subspace, hence $\lambda X$ is an extension of $X$. In fact, the function $i\from X\to \lambda X$ defined by
$$
    i(x) = \{x\}^+ = \{A\con X : A\mbox{ is closed and $x\in A$}\}
$$
is an embedding of $X$ into $\lambda X$. We identify each point $x\in X$ with the mls $i(x)$. The closure of $X$ in $\lambda X$ is the familiar \emph{Wallman compactification} of $X$.

A \emph{defining set} for $\mathcal{M}\in \lambda X$ is a subset $S$ of $X$ with the following property: for every $M\in \mathcal{M}$ there exists $M'\in \mathcal{M}$ such that $M'\con M\cap S$. An mls $\mathcal{M}\in\lambda X$ is called \emph{finitely generated} (or an \emph{fmls}) if it has a finite defining set. The subspace of $\lambda X$ consisting of all fmls's is denoted by $\lambda_f(X)$.

Clearly, $i(x)$ is an fmls whenever $x\in X$, having $\{x\}$ as defining set. Thus we have $X \con \lambda_f(X)$.

We shall need the following results that were proved for $\lambda_f(X)$ by Verbeek~\cite{Verbeek72}.

\begin{enumerate}
\item[(V1)] (\cite[IV.3.4(iii)]{Verbeek72}) If $X$ is $T_2$ then $X$ is closed in $\lambda_f(X)$.
\item[(V2)] (\cite[IV.3.4(v)+(vi)]{Verbeek72}) If $X$ is $T_2$ then so is $\lambda_f(X)$. Similarly for $T_{3.5}$.
\item[(V3)] (\cite[IV.3.4(viii)]{Verbeek72}) If $X$ is connected then $\lambda_f(X)$ is both connected and locally connected.
\item[(V4)] (\cite[III.2.5(b)]{Verbeek72}) $\lambda_f(X)$ can be represented as the countable union of subspaces each of which is a continuous image of some finite power of $X$.
\item[(V5)] (\cite[III.4.3(iv)]{Verbeek72}) $\density(X) = \density(\lambda_f(X))$.
\end{enumerate}

We shall need the fact that if $X$ is $T_3$ then so is $\lambda_f(X)$.
This is not stated explicitly in Verbeek~\cite{Verbeek72}, hence we provide a (simple) proof.

\begin{lemma}\label{IsT_3}
If $X$ is $T_3$ then so is $\lambda_f(X)$.
\end{lemma}

\begin{proof}
Let $\mathcal{M}\in \lambda_f(X)$, and let $G\con \lambda_f(X)$ be open such that $\mathcal{M} \in G$. In addition, let $S$ be a finite defining set for $\mathcal{M}$. Since $G$ is open, there is a finite collection $\mathcal{U}$ of open subsets of $X$ such that $\mathcal{M}\in \bigcap_{U\in \mathcal{U}} U^+\con G$. For every $U\in \mathcal{U}$, there exists $F(U)\con U\cap S$ such that $\mathcal{M}\in \bigcap_{U\in \mathcal{U}} F(U)^+$. For every $U\in \mathcal{U}$ let $V(U)$ be an open subset of $X$ such that $F(U) \con V(U)\con \CL{V(U)} \con U$.( Here we use that $X$ is $T_3$.) Then $\bigcap_{U\in \mathcal{U}} \CL{V(U)}^+$ is a closed neighborhood of $\mathcal{M}$
in $\lambda_f(X)$ that is contained  $G$.
\end{proof}

\begin{lemma}\label{pdvanlambda}
For any $T_2$ space $X$ we have $\pd(X) \le \pd(\lambda_f(X))$. If, in addition, $X$ is neat then $\pd(X) = \pd(\lambda_f(X))$.
\end{lemma}

\begin{proof}
We first prove that $\pd(X)\le \pd(\lambda_f(X))$. This is trivial if $X$ is finite, so we assume that $X$ is infinite. Note that then $\lambda_f(X)$ and hence $\pd(\lambda_f(X))$ are also infinite.
Let $U\from X\to \tau(X)$ be any neighborhood assignment on $X$. Define $V\from X\to \tau(\lambda_f(X))$ by $V(x) = U(x)^+$. Extend $V$ to a neighborhood assignment $W$ on $\lambda_f(X)$ by putting $W(x) = V(x)$ for
$x\in X$ and $W(\mathcal{M}) = \lambda_f(X)\setminus X$ for $\mathcal{M}\in \lambda_f(X)\setminus X$. Then there is a subset $A$ of $\lambda_f(X)$ that meets every element of the collection $\{V(x) : x\in X\}$ and has size  $\pd(\lambda_f(X))$. For every $\mathcal{M}\in A$ let $F(\mathcal{M})$ be a finite defining set for $\mathcal{M}$. Put
$$
    B = \bigcup\{F(\mathcal{M}) : \mathcal{M}\in A\}.
$$
Then $|B|\le \omega{\cdot}|A| = |A|$, and we claim that $B$ meets $U(x)$ for each $x\in X$. Indeed, take any $\mathcal{M}\in A\cap U(x)^+$.
This means that there is a subset $G$ of $F(\mathcal{M})$ such that $G\in \mathcal{M}$ and $G\con U(x)$. Hence $\emptyset \ne G \con U(x) \cap B$, and this completes the proof.

Now assume that $X$ is also neat. First observe that then $\pd(X^n)=\pd(X)$ for every $n <\omega$ by Corollary ~\ref{X^n}. Hence we get $\pd(\lambda_f(X)) \le \pd(X)$ applying (V4) and Lemma~\ref{eerstelemma}.
\end{proof}

We now describe a very general version of the well-known `cone' construction that can be used to obtain
connectifications of spaces in a natural way. The input of the construction consists of an
arbitrary topological space $X$ and an infinite connected $T_1$ space $P$ with a distinguished point $p \in P$. The output is the space
$Z = Z(X,P,p)$ whose underlying set is
$$
    {\big (}X\times (P\setminus \{p\}){\big )} \cup \{p\},
$$
where, of course, $p \not\in X\times (P\setminus \{p\})$.
The topology of $Z$ is defined as follows:
Basic neighborhoods of points of $X\times (P\setminus \{p\})$ in $Z$ are just the standard product neighborhoods. A basic neighborhood of $p$ in $Z$ has the form
$$
    {\big (}X\times (U\setminus\{p\}){\big )}\cup \{p\},
$$
where $U$ is any neighborhood of $p$ in $P$. It is easy to check that this is indeed a topology.
It is also straightforward to show for every point $q \in P \setminus \{p\}$ that
$X \times \{q\}$ is a closed homeomorphic copy of $X$ in $Z$ and that for every $x \in X$
the subspace $(\{x\} \times P\setminus \{p\}) \cup \{p\}$ is a homeomorphic copy of $P$ in $Z$,
hence $Z$ is connected. We leave it to the reader to check these facts as well as the following
proposition.

\begin{proposition}\label{Z}
If both $X$ and $P$ are $T_2$ (or $T_3$, or $T_{3.5}$) then so is $Z(X,P,p)$.
\end{proposition}

The following lemma provides us with a procedure that transforms any $T_2$ (resp. $T_3$)
pd-example into a connected and locally connected $T_2$ (resp. $T_3$) pd-example.
This, of course, will complete the proof of part (i) of Theorem~\ref{maintheorem1}.

\begin{lemma}\label{maandageen}

\begin{enumerate}

\item Let $X$ be an infinite neat $T_2$ space and $P$ be a countably infinite connected $T_2$-space. Then for any $p \in P$, the $T_2$ space
$Y = \lambda_f(Z(X,P,p))$ is connected and locally connected, moreover $\density(X)=\density(Y)$ and $\pd(X)=\pd(Y)$.
\item If $X$ is an infinite neat $T_3$ space and $P$ is a separable connected $T_2$-space
of cardinality $\omega_1$ then for any $p \in P$ the $T_3$ space
$Y = \lambda_f(Z(X,P,p))$ is connected and locally connected, moreover $\density(X)=\density(Y)$ and $\pd(X)=\pd(Y)$.
\end{enumerate}
\end{lemma}

\begin{proof}
For (1), fix a countably infinite connected $T_2$-space $P$  and a point $p\in P$.
(The existence of such a space $P$ was first proved by Urysohn~\cite{Urysohn25}.)
Then $Z = Z(X,P,p)$ is connected, $T_2$ and, since $P$ is countable, we clearly have $\density(X)=\density(X\times P) = \density(Z)$.

Next we show that $\pd(Z) = \pd(X)$. It is obvious that $\pd(Z) = \pd(X\times (P\setminus \{p\}))$.
But by Lemma~\ref{vrijdageen} and because $X$ a continuous image of $X\times (P\setminus \{p\})$, we have $\pd(X) = \pd(X\times (P\setminus \{p\}))$.
Since $X$ is neat and $P$ is countable it is obvious that $Z$ is neat as well.

The superextension $Y = \lambda_f(Z)$ is a connected and locally connected $T_2$ space by (V1) and (V3).
Moreover, $\density(\lambda_f(Z))=\density(Z) = \density(X)$ by (V5). Since $Z$ is neat, we get by the above and by Lemma~\ref{pdvanlambda} that $\pd(X)=\pd(Z)=\pd(\lambda_f(Z))$, hence we are done.

For (2) we have to do a little more work. It is well-known that there is a connected $T_3$  space of size $\omega_1$.
This is the smallest such cardinality possible since every $T_3$ countable space is 0-dimensional. An example is Hewitt's Condensed Corkscrew and there are others.
For details, see Steen and Seebach~\cite[p.\ 111]{SteenSeebach95}.

However, for our construction we need a \emph{separable} connected  $T_3$ space $P$ of size $\omega_1$.
Luckily for us, it was recently proved by Ciesielski and Wojciechowsk~\cite[Theorem 8]{CiesielskiWojciechowski16} that there is such a space.
Observe that then $P$ is neat. For let us assume that it contains a countable nonempty open subset $U$.
Then it contains by regularity a nonempty open subset $V$ such that $\CL{V}\con U$. Hence $V$ is a countable regular space, and so is 0-dimensional.
But then $V$ contains a nonempty subset $C$ which is clopen in $\CL{V}$. But this $C$ would be clopen in $P$.
This argument actually yields that for every non-singleton connected $T_{3}$ space $X$ we have $\Delta(X) \ge \omega_1$.

Now, as in (1), consider $Z = Z(X,P,p)$ for some $p \in P$. Since both $X$ and $P$ are neat,
it is not difficult to check that $Z$ is neat as well.
Since $P$ is separable, we obviously have $\density(X)=\density(X\times P) = \density(Z)$.

We next show that $\pd(Z) = \pd(X)$. If $X$ is countable then $X\times P$ is separable, hence so is $Z$. But then $\pd(Z) = \pd(X)=\omega$. If $|X|\ge \omega_1$ then, since $|P|=\omega_1$ and $P$ is separable,
we may apply Lemma~\ref{vrijdageen} to obtain $\pd(X \times (P \backslash \{p\}))=\pd(X)$. But we also have $\pd(Z) = \pd(X \times (P \backslash \{p\}))$,
hence we are done.

Now, the space $Y = \lambda_f(Z)$ is as we want by following the same argumentation as in part (1). We, of course, have to use that $\lambda_f(Z)$ is $T_3$ by Lemma~\ref{IsT_3}.
\end{proof}

This completes the proof of part (I) of Theorem \ref{maintheorem1}. The following results explain why one cannot replace in it $T_3$ with $T_{3.5}$.

\begin{theorem}\label{Delta}
Assume that \WGCH\ holds from the cardinal $\kappa$ on, i.e. $2^\lambda < \lambda^{+\omega}$ for all $\lambda \ge \kappa$.
Then for every $T_2$ pd-example $X$ we have $\Delta(X) < \kappa^{+\omega}$.
If, in addition, \GCH\ holds in the interval $[\kappa,\, \kappa^{+\omega})$,  i.e. $2^\lambda = \lambda^+$ for all $\lambda$ with $\kappa \le \lambda < \kappa^{+\omega}$,
and $X$ is $T_3$ then we even have $\Delta(X) \le \kappa$.
\end{theorem}

\begin{proof}
Assume, on the contrary, that $\Delta(X) \ge \kappa^{+\omega}$. We may also assume, without any loss of generality, that $X$ is neat because,
by Proposition \ref{neat}, we can replace $X$ with its open subspace $Y$ that is a neat pd-example and, of course, satisfies $\Delta(Y) \ge \Delta(X)$.

But then $|X| = \Delta(X) = \mu^{+n}$, where $\mu \ge \kappa^{+\omega}$ is a limit and hence a strong limit cardinal.
This, however, contradicts  \cite[Theorem 3.1]{JuhaszSoukupSzentmiklossy16} which says that in this case $\pd(X) = \density(X)$.

To see the second part and striving for a contradiction, assume that $X$ is a $T_3$ pd-example with $\Delta(X) > \kappa$.
Again we may also assume, without any loss of generality, that $X$ is neat, hence we have $$\pd(X) < \density(X) \le |X| = \Delta(X) < \kappa^{+\omega}.$$
By \cite[Theorem 4.8]{JuhaszSoukupSzentmiklossy16} $\density(X) < 2^{\pd(X)}$ holds for every $T_3$ space $X$,
hence our cardinal arithmetic assumptions imply $\pd(X) < \kappa$. But then we also have $2^\pd(X) \le 2^\kappa = \kappa^+$,
while \cite[Theorem 4.9]{JuhaszSoukupSzentmiklossy16} says that any $T_3$ pd-example $X$ also satisfies $\Delta(X) < 2^\pd(X)$.
Hence we obtain $\Delta(X) \le \kappa$, a contradiction.
\end{proof}

Since for every non-singleton connected $T_{3.5}$ space $X$ we have $\Delta(X) \ge \mathfrak{c}$, we immediately obtain the following.

\begin{corollary}\label{ctT3.5}
If $\mathfrak{c} = \kappa^+ > \aleph_\omega$, moreover \WGCH\ holds from $\kappa$ on and \GCH\ holds in the interval $[\kappa,\, \kappa^{+\omega})$,
then $\mathbf{S} \ne \emptyset$,
in fact $\aleph_\omega \in \mathbf{S}$, but there is no connected $T_{3.5}$ pd-example.
\end{corollary}

Let us remark that the assumptions of Corollary \ref{ctT3.5} are satisfied in a generic extension obtained by
adding $\kappa^+$ Cohen reals to a model of \GCH\ for any $\kappa \ge \aleph_\omega$.

\smallskip

If $X$ is any $T_{3.5}$ space with $|X| = \Delta(X) \ge \big|[0,1] \big| = \mathfrak{c}$, then clearly the ordinary cone
$Z = Z(X,[0,1],1)$ is a connected $T_{3.5}$ space that is neat and, by Lemma \ref{vrijdageen} satisfies both $\pd(Z) = \pd(X)$
and $\density(Z) = \density(X)$. Consequently, then $Y = \lambda_f(Z)$ is a connected and locally connected $T_{3.5}$ pd-example
if $X$ is a pd-example. Now, the existence of such an $X$ follows from $\mathbf{S} \setminus \mathfrak{c} \ne \emptyset$,
however we shall see later that this same assumption gives us much stronger $T_{3.5}$ pd-examples.

\bigskip

\section{Proof of Theorem~\ref{maintheorem1} part (II): topological group pd-examples}\label{functortwee}

\bigskip

The idea of the proof of part (II) of Theorem~\ref{maintheorem1} is very simple: We show that
if $X$ is a neat $T_{3.5}$ pd-example then $A(X)$, the free abelian topological group on $X$
is a pd-example as well. Now let us see the details.

\par\medskip
\noindent {\bf{ Free topological groups:}}
If $X$ is a $T_{3.5}$ space, then $F(X)$ and $A(X)$ denote the free topological group and the free abelian topological group on $X$. That is, $F(X)$ is a topological group containing (a homeomorphic copy of) $X$ such that
\begin{enumerate}
\item $X$ generates $F(X)$ algebraically,
\item every continuous function $f\from X\to H$, where $H$ is any topological group, can be extended to a continuous homomorphism $\bar f\from F(X)\to H$.
\end{enumerate}
Similarly for $A(X)$. The existence of these groups was proved by Markov~\cite{Markov41}. See \arhang\ and Tkachenko~\cite[Chapter 7]{ArhangTkachenko08} for details and references. It is known that
\begin{enumerate}
\item[\rm{(FG1)}] (\cite[Theorem 7.1.13]{ArhangTkachenko08}) $X$ is closed in $F(X)$ as well as $A(X)$,
\item[\rm{(FG2)}] (\cite[Theorem 7.1.5]{ArhangTkachenko08}) $F(X)$ and $A(X)$ are $T_{3.5}$ (being $T_2$ topological groups),
\item[\rm{(FG3)}] (\cite[Theorem 7.1.13]{ArhangTkachenko08}) $F(X)$ as well as $A(X)$ can be represented as a countable union of subspaces each of which is a continuous image of some finite power of $X$.
\end{enumerate}

The first part of the following crucial result is probably well-known.

\begin{proposition}\label{eerstezondag}
Let $X$ be infinite and $T_{3.5}$. Then $\density(X) = \density(F(X))=\density(A(X))$ and if $X$ is neat, then $\pd(X) = \pd(F(X))=\pd(A(X))$.
\end{proposition}

\begin{proof}
We will only check this for $F(X)$, the proof for $A(X)$ is similar. That $\density(F(X))\le \density(X)$ is a direct consequence of (FG3). Now let $D$ be dense in $F(X)$. Every element of $D\setminus \{e\}$, where $e$ is the neutral element of $F(X)$, can be written uniquely in the form
$$
    d = x_1^{r_1}x_2^{r_2}\cdots x_n^{r_n}, \leqno{(\dagger)}
$$
where $n\ge 1$, $r_i\in\integers\setminus \{0\}$, $x_i\in X$, and $x_i\not= x_{i+1}$ for every $i=1,2,\dots,n{-}1$. Let $E$ be the set of all
points $x \in X$ that appear in the expressions $(\dagger)$. We claim that $E$ is dense in $X$. Indeed, assume that this is not true, and fix a continuous function $f\from X\to \reals$ such that $f(\CL{E})\con \{0\}$ and $f(x) = 1$, for some $x\in X\setminus \CL{E}$. We can extend $f$ to a continuous homomorphism $\bar f\from F(X)\to \reals$. Then $\bar f(D) \con \{0\}$, but $\bar f$ is not constant. This is clearly a contradiction. Hence $\density(X)\le \density(F(X))$.

Now assume that $X$ is neat. Then $\pd(X^n) = \pd(X)$ for every $n < \omega$ (Corollary~\ref{X^n}). Hence by (FG3) and Lemma~\ref{eerstelemmaA} we get $\pd(F(X)) \le \pd(X)$. For the converse inequality, let $U\from X\to \tau^+(X)$ be any neighborhood assignment on $X$. For every $x\in X$, let $f_x\from X\to [0,1]$ be continuous such that $f_x(x) = 1$ and $f_x(X\setminus U(X)) \con \{0\}$. Consider the continuous homomorphism $\bar f_x\from F(X)\to \reals$ that extends $f$. Put $V(x) = \bar f_x^{-1}((0,2))$. Then $V\from X\to \tau^+(F(X))$ is a neighborhood assignment, which can be extended to a full neighborhood assignment by simply putting $V(y) = F(X)$ for every $y \in F(X)\setminus X$. Now let $B$ be a subset of $F(X)$ of size at most $\pd(F(X) )$ meeting every set of the form $V(x)$, for $x\in X$. We can write every $b\in B\setminus\{e\}$ uniquely in the form
$$
    b = y_1^{r_1}y_2^{r_2}\cdots y_n^{r_n}, \leqno{(\ddagger)}
$$
where $n\ge 1$, $r_i\in\integers\setminus \{0\}$, $y_i\in X$, and $y_i\not= y_{i+1}$ for every $i=1,2,\dots,n{-}1$. Let $F$ be the set of all $y$'s that appear in the expressions $(\ddagger)$. Then, clearly, $|F|\le \pd(F(X))$. We claim that $F\cap U(x)\not=\emptyset$ for every $x\in X$. Indeed, striving for a contradiction, assume that $F\cap U(x)=\emptyset$ for some $x\in X$. Then $f_x(F)\con \{0\}$, and so $\bar f_x(B\setminus \{e\})\con \{0\}$. But this is a contradiction, since $B\setminus\{e\}$ meets $V(x) = \bar f_x^{-1}((0,2))$. So we conclude that $\pd(X)\le \pd(F(X))$, as desired.
\end{proof}

To prove part (II) of Theorem~\ref{maintheorem1}, we simply have to recall that, by  Theorem~\ref{maintheorem1} and Proposition \ref{neat}, if $\mathbf{S} \ne \emptyset$
then there is a neat 0-dimensional $T_2$, hence $T_{3.5}$ pd-example. Then we may apply Proposition \ref{eerstezondag} to obtain from $X$
the $T_2$ abelian topological group pd-example $A(X)$.

\bigskip

\section{Proof of Theorem~\ref{maintheorem2}: Connected topological group pd-examples}

\bigskip

We have seen in the previous section that if $X$ is a neat $T_{3.5}$ pd-example then the free (abelian) topological group over $X$
is a pd-example as well. In fact, what we could show was that this construction preserves the values of both $\pd(X)$ and $\density(X)$.
To obtain a proof of Theorem~\ref{maintheorem2}, we compose this construction with a procedure due to Hartman and Mycielski that, in turn
embeds any (abelian) $T_2$ topological group $G$ in a larger (abelian) $T_2$ topological group $G^\bullet$ that is pathwise connected and locally pathwise connected.
Under certain conditions this procedure also preserves the values of both $\pd(G)$ and $\density(G)$ and, for an appropriate neat $T_{3.5}$ pd-example $X$,
then $A(X)^\bullet$ is a pathwise connected and locally pathwise connected abelian $T_2$ topological group pd-example.

\medskip

We now describe the Hartman-Mycielski construction, following its presentation in~Ar-hangel'skii and Tkachenko~\cite[3.8.1]{ArhangTkachenko08},
and then prove a few new results about it that will be needed.
All topological groups we consider are assumed to be $T_2$.

Let $G$ be a topological group with neutral element $e$ and with group operation written multiplicatively. $G^\bullet$ is defined to be the set
of all step functions $f : J=[0,1) \to G$ such that, for some sequence $0= a_0 < a_1 < \cdots < a_n = 1$, the function $f$ is constant on $[a_k,a_{k+1})$ for every $k = 0, \dots, n{-}1$.
Define a binary operation $*$ on $G^\bullet$ by $(f*g)(x) = f(x)\cdot g(x)$ for all $f,g\in G^\bullet$ and $x\in G$.
Then every $f\in G^\bullet$ has a unique inverse in $G^\bullet$, defined by $f^{-1}(x) = (f(x))^{-1}$.
Then $(G^\bullet, *)$ is a group with identity $e^\bullet$, where $e^\bullet (r) = e$ for each $r\in J$.
It is also easy to see that $G$ can be identified with a subgroup of $G^\bullet$ via $x\mapsto x^\bullet$, where $x^\bullet(r) = x$ for every $r\in J$.

Let $V$ be a neighborhood of $e$ in $G$, and for every $\varepsilon > 0$, put
$$
    O(V,\varepsilon) = \{f\in G^\bullet : \mu(\{r\in J : f(r) \not\in V\})\} < \varepsilon,
$$
here $\mu$ denotes Lebesgue measure.
The $O(V,\varepsilon)$ are the neighborhoods of the neutral element $e^\bullet$ of $G^\bullet$ that generate its group topology. The following facts are known:

\begin{enumerate}
\item[(HM1)] (\cite[3.8.2]{ArhangTkachenko08}) $G^\bullet$ is a topological group and is pathwise connected and locally pathwise connected.
\item[(HM2)] (\cite[3.8.3]{ArhangTkachenko08}) The function $i_G\from G \to G^\bullet$ defined by $i_G(x) = x^\bullet$ is a topological isomorphism of $G$ onto a closed subgroup of $G^\bullet$.
\item[(HM3)] (\cite[3.8.8(e)]{ArhangTkachenko08}) $\density(G^\bullet) \le \density(G)$.
\end{enumerate}

We next prove that that in (HM3) one actually has equality.

\begin{lemma}\label{eerstelemma}
$\density(G) \le \density(G^\bullet)$.
\end{lemma}

\begin{proof}
Let $D$ be dense in $G^\bullet$, and put $E= \bigcup_{f\in D} f(J)$. Since the range of every element of $D$ is finite, $|E|=|D|$. We claim that $E$ is dense in $G$. If not, let $V$ be an open neighborhood of $e$ and let $x\in G$ be such that $xV\cap E=\emptyset$. The nonempty open subset $x^\bullet * O(V,\frac{1}{2})$ of $G^\bullet$ meets $D$, say in the point $f$. Let $g\in O(V,\frac{1}{2})$ be such that $f = x^\bullet * g$. Clearly, $\range(g)\cap V\not=\emptyset$, say $g(t)\in V$. Then $f(t) = x^\bullet(t) g(t)=xg(t)\in xV\con G\setminus E$. This is a contradiction.
\end{proof}

\begin{corollary}\label{eerstecor}
$\density(G)=\density(G^\bullet)$.
\end{corollary}

To obtain conditions under which we also have $\pd(G)=\pd(G^\bullet)$ we first make a little detour.

\begin{lemma}\label{tweedelemma}
If the infinite $T_{3.5}$ space $X$ is neat then so are both $A(X)$ and $F(X)$.
\end{lemma}

\begin{proof}
We prove this for $F(X)$, the proof for $A(X)$ being entirely similar. First observe that by (FG3) from section \ref{functortwee}, we have
$|F(X)| = |X|$. Let $U$ be any open subset of $F(X)$ containing the neutral element of $F(X)$.
Fix $x\in X$, and observe that $xU$ contains $x$, hence $xU\cap X$ is a nonempty open subset of $X$.
This gives us that $|U|\le |F(X)| = |X|$ on one hand and $|X| = |xU\cap X| \le |U|$ on the other, hence we are done.
\end{proof}

The following is our crucial result concerning the Hartman-Mycielski construction.

\begin{lemma}\label{derdelemma}
If $G$ is neat and $|G|\ge \cont$ then $G^\bullet$ is neat and $\pd(G^\bullet) = \pd(G)$.
\end{lemma}

\begin{proof}
That $G^\bullet$ is neat follows by the same argument as in the proof of Lemma~\ref{tweedelemma}.
The assumption $|G|\ge \cont$ ensures that we have $|G| = |G^\bullet|$.

Let $U\from G\to \tau(G)$ be a neighborhood assignment on $G$. For every $x\in G$ we have then a neighborhood $V_x$ of the neutral element $e$ of $G$ such that $xV_x \con U(x)$. Observe that $V\from G\to \tau(G^\bullet)$ defined by $V(x) = x^\bullet * O(V_x,\frac{1}{2})$ is a partial neighborhood assignment on $G^\bullet$ (note that we identify $x$ with $x^\bullet$).
Let $A$ be a subset of $G^\bullet$ of size $\pd(G^\bullet)$ such that $A\cap V(x) \not=\emptyset$ for every $x\in G$ and put $B=\bigcup_{f\in A} f(J)$.
Then $|B| \le |A|$ because  $A$ is infinite and $f(J)$ is finite for all $f \in A$.
Now take an arbitrary point $x\in G$ and let $f\in A$ be such that $f\in V(x) =  x^\bullet * O(V_x,\frac{1}{2})$.
There is $g\in O(V_x,\frac{1}{2})$ such that $f = x^\bullet * g$. Clearly, then $\range(g) \cap V_x\not= \emptyset$, say $g(t)\in V_x$.
But then $f(t) = x(t) g(t) \in x V_x\con U(x)$, hence $B$ meets $U(x)$. This proves that $\pd(G) \le \pd(G^\bullet)$.
Note that this part only used that $G$ is infinite.
In the proof of the reverse inequality, however, the assumption $|G|\ge \cont$ that ensures $|G| = |G^\bullet|$ will play an essential role.

Let $U\from G^\bullet\to \tau(G^\bullet)$ be a neighborhood assignment on $G^\bullet$. For every $f\in G^\bullet$ we may then
take a neighborhood $V_f$ of $e$ in $G$ and an $\varepsilon_f > 0$ such that $f * O(V_f,\varepsilon_f) \con U(f)$. Consider the collection
$$
    \mathcal{V}=\{x V_f : f\in G^\bullet, x\in G\}
$$
of open subsets of $X$ which, using $|G|\ge \cont$, has size at most $|G|$. Since $G$ is neat, $\mathcal{V}$
can be pinned down by an infinite set $D \subset G$ of size at most $\pd(G)$. Now let $S$ be the set of all $g$ in $G^\bullet$
for which there exist for some $n$ rational numbers $0=b_0 < b_1 < \cdots <b_{n-1} < b_n = 1$ and elements $d_0, \dots, d_{n-1}\in D$ such that
$g$ takes the constant value $d_k$ on $[b_k,b_{k+1})$ for every $k = 0, \dots, n{-}1$. Clearly, we have then $|S| =|D|$.

For any fixed $f\in G^\bullet$ there exist numbers $0= a_0 < a_1 < ... < a_n = 1$ such that
the function $f$ is constant $x_k$ on $[a_k, a_{k+1})$ for each $0\le k < n$. We may then choose rational numbers $b_1, \dots, b_{n-1}$
such that $a_k \le b_k < a_{k+1}$ for each $1\le k < n$ and $\sum_{k=1}^{n-1}(b_k - a_k) < \varepsilon_f$. Put $b_0=0$ and $b_n = 1$.
For every $0\le k < n$, we can choose a point $y_k \in D \cap  f(a_k) V_f$, and then define an
element $g \in S$ by letting $g(r) = y_k$ for each $r \in [b_k, b_{k+1} )$, $0\le k < n$. We claim that then $g \in f * O(V_f,\varepsilon_f)$.
To see this, it suffices to prove that the function $h\from J\to G$ defined by $h(t) = f(t)^{-1} g(t)$ belongs to $O(V_f,\varepsilon_f)$, but this is clear from our construction.
Thus we have shown that $S$ meets $f * O(V_f,\varepsilon_f) \subset U(f)$ for all $f\in G^\bullet$, proving that $\pd(G^\bullet) \le \pd(G)$.
\end{proof}

\begin{corollary}
Let $X$ be a neat $T_{3.5}$ space such that $|X|\ge \cont$. Then $X$ admits a closed embedding into a $T_{3.5}$ topological group $G$ such that
\begin{enumerate}
\item $\density(X) = \density(G)$,
\item $\pd(X) = \pd(G)$,
\item $G$ is neat,
\item $G$ is pathwise connected and locally pathwise connected.
\end{enumerate}

In particular, if $X$ is a neat $T_{3.5}$ pd-example of size $\ge \mathfrak{c}$
then $A(X)^\bullet$ is a pathwise connected and locally pathwise connected
abelian topological group pd-example.
\end{corollary}

But \cite[Theorem 3.3]{JuhaszSoukupSzentmiklossy16} implies that if there is $\mu \in \mathbf{S}$
with $\mu \ge \mathfrak{c}$ then there is a neat 0-dimensional $T_2$, hence $T_{3.5}$ pd-example of size $\ge \mathfrak{c}$,
completing the proof of the first sentence in Theorem~\ref{maintheorem2}.

To prove the second sentence, assume that $X$ is a connected $T_{3.5}$  pd-example and recall that then
$\Delta(X) \ge \mathfrak{c}$. But then by Proposition \ref{neat} there is a neat open subspace $Y \con X$
that is also a pd-example. We do not know if $Y$ is connected but we do know that $$|Y| = \Delta(Y) \ge \Delta(X) \ge \mathfrak{c}.$$

But \cite[Theorem 3.5]{JuhaszSoukupSzentmiklossy16} and the remark made after it imply that if SSH holds
then $|Y|$ must be singular, while \cite[Theorem 3.1]{JuhaszSoukupSzentmiklossy16} implies that $|Y|$ cannot be strong limit, hence $|Y| \in \mathbf{S}$. Thus we have completed the proof of the second and third
sentences in Theorem~\ref{maintheorem2}.

Let us conclude this section with a couple of remarks. First, note that in Corollary \ref{ctT3.5} we formulated
some simple cardinal arithmetic conditions that imply $\mathbf{S} \ne \emptyset$ but $\mu < \mathfrak{c}$
for all $\mu \in \mathbf{S}$. Secondly, we note that large cardinals, hence going beyond ZFC are needed
to refute SSH.

\bigskip

\section{Proof of Theorem~\ref{maintheorem3}}

\bigskip

For every infinite $T_{3.5}$ space $X$ one can define the free locally convex $\mathbb{R}$-vector space $L(X)$ on $X$.
This is a space with similar properties as the free groups that we considered in section \ref{functortwee}.
The space $L(X)$ contains $X$ as a closed subspace and at the same time $X$ forms an $\mathbb{R}$-vector space
basis for $L(X)$. Moreover, the following defining property holds: every continuous mapping $f$ from $X$ to a locally
convex $\mathbb{R}$-vector space $E$ can be extended to a continuous linear operator $\bar f\from L(X)\to E$. The existence and uniqueness of $L(X)$ was proved by Markov in~\cite{Markov41}.

We can treat $L(X)$ in almost the same way as we treated the free topological groups $F(X)$ and $A(X)$. There is one important difference however: the statement (FG3) should be replaced by the following.
\begin{enumerate}
\item[\rm{(FLC3)}] $L(X)$  can be represented as the countable union of subspaces each of which is a continuous image of some finite power of $X\times\reals$.
\end{enumerate}

So, to use Lemma~\ref{vrijdageen} to conclude $\pd(X) = \pd(X\times\reals)$ we need $|X| = \Delta(X) \ge \mathfrak{c}$.

\begin{proposition}\label{tweedezondagNIEUW}
Let $X$ be an infinite $T_{3.5}$ space. Then $\density(X) = \density(L(X))$ and if, in addition, $X$ is neat and $|X|\ge \cont$ then $\pd(X) = \pd(L(X))$.
\end{proposition}

\begin{proof}
That $\density(L(X)) \le \density(X)$ is a direct consequence of (FLC3). The proof that $\density(X) \le \density(L(X))$ is completely analogous to the proof that $\density(X) \le \density(F(X))$ in Proposition~\ref{eerstezondag}. Now assume that $X$ is neat and that $|X| \ge \cont$. Then $\pd(X\times\reals ) = \pd(X)$ by Lemma~\ref{vrijdageen}. Since $X\times\reals$ is clearly neat, we have by Corollary~\ref{X^n} that $\pd((X\times\reals)^n)=\pd(X\times \reals) = \pd(X)$ for every $0 < n < \omega$. Hence again by (FLC3) we get $\pd(L(X)) \le \pd(X)$. That $\pd(X) \le \pd(L(X))$ follows exactly as in the proof of Proposition~\ref{eerstezondag} for $F(X)$.
\end{proof}

Hence if there is a $T_{3.5}$ pd-example $X$ such that $\pd(X) \ge \cont$ then $L(X)$ is a locally convex $\mathbb{R}$-vector space that is also a pd-example.
Such a space $X$ exists by the construction given in \cite[Theorem 3.3]{JuhaszSoukupSzentmiklossy16}, provided that there is a cardinal $\mu \in \mathbf{S}$
with $\mu \ge \mathfrak{c}$.
The proof of Theorem~\ref{maintheorem3} is thus completed.


\def\cprime{$'$}
\makeatletter \renewcommand{\@biblabel}[1]{\hfill[#1]}\makeatother

\end{document}